\documentclass[10pt]{article}
\usepackage{latexsym,amssymb,amsmath,epsfig}
\usepackage{epsfig,graphicx,latexsym,graphics} 
\usepackage{float}
\usepackage{amsfonts}
\usepackage{color}
\newtheorem{thm}{Theorem}
\newtheorem{proof}{Proof}
\newtheorem{lemma}{Lemma}
\newtheorem{cor}{Corollary}

\begin{document}
\begin{center}
{\Large \bf Some Bounds on Zeroth-Order General Randi\'c Index}\\
\vspace{6mm}
{ Muhammad Kamran Jamil$^{0,1,4}$, Ioan Tomescu$^2$, Muhammad Imran$^{1,3,*}$}\\
\vspace{4mm}
{\it $^1$Department of Mathematical Sciences, United Arab University, Al Ain, United Arab Emirates.}\\
{\it $^2$Faculty of Mathematics and Computer Science, University of Bucharest, Bucharest, Romania.}\\
{\it $^3$Department of Mathematics, School of Natural Sciences, National University of Science and Technology, Islamabad.}\\
{\it $^4$Department of Mathematics, Riphah Institute of Computing and Applied Sciences, Riphah International University, Lahore, Pakistan.}\\
\vspace{2mm}
{m.kamran.sms@gmail.com, ioan@fmi.unibuc.ro, imrandhab@gmail.com}
\vspace{6mm}\\
\end{center}
\footnotetext{Corresponding author}

\textbf{Abstract:} For a graph $G$ without isolated vertices, the inverse degree of a graph $G$ is defined as $ID(G)=\sum_{u\in V(G)}d(u)^{-1}$ where $d(u)$ is the number of vertices adjacent to the vertex $u$ in $G$. By replacing $-1$ by any non-zero real number we obtain zeroth-order general Randi\'c index, i.e. $^0R_{\gamma}(G)=\sum_{u\in V(G)}d(u)^{\gamma}$ where $\gamma$ is any non-zero real number. In \cite{xd}, Xu et. al. determined some upper and lower bounds on the inverse degree for a connected graph $G$ in terms of chromatic number, clique number, connectivity, number of cut edges. In this paper, we extend their results and investigate if the same results hold for $\gamma<0$. The corresponding extremal graphs have been also characterized.

\section{Introduction}

Throughout this paper, we only consider finite, connected and simple graphs and for the terminologies on the graph theory not defined here one can see \cite{bm}.  Let $G$ be a graph with vertex set $V(G)$ and edge set $E(G)$. The number of elements in $V(G)$ and $E(G)$ are called the {\it order} and the {\it size} of $G$, respectively. For a vertex $u$ of $G$, $N_G(u)$ is the set of vertices adjacent to the vertex $u$ in $G$ and the number of elements in $N_G(u)$ is called the {\it degree} of the vertex $u$ in $G$, denoted by $d_G(u)$ or simply $d(u)$. Also, $N_G[u]=N_G(u)\cup \{u\}$. A vertex $u$ is said to be a {\it pendant vertex} if $d(u)=1$ and an edge is said to be a {\it pendant edge} if it is incident with a pendant vertex. In a graph $G$ the {\it maximum} and {\it minimum} degrees are denoted by $\triangle(G)$ and $\delta(G)$, respectively. For a subset $S$ of $V(G)$, $G-S$ is a subgraph obtained from $G$ by deleting the vertices of $S$ and the edges incident with them. Similarly, if we have a subset $T$ of $E(G)$, then $G-T$ is the graph obtained by deleting the edges of $T$. For two non-adjacent vertices $u$ and $v$ in a graph $G$, $G+uv$ is the graph obtained from $G$ by adding an edge between $u$ and $v$ and $G-uv$
is the graph deduced by deleting the edge $uv$. \\
The minimum number of colors required to color a graph $G$ in such a way that no two adjacent vertices have the same color is called the {\it chromatic number} of $G$, it is denoted by $\chi(G)$. A {\it clique} of a graph $G$ is a subset $V'$ of $V(G)$ such that in $G[V']$, the subgraph of $G$ induced by $V'$, is a complete graph. The maximum number of vertices in a clique is called the {\it clique number} of $G$, it is denoted by $\omega(G)$. Let $G_1$ and $G_2$ be two vertex disjoint graphs. $G_1\cup G_2$ is the graph which consists of two components $G_1$ and $G_2$. The {\it join} of $G_1$ and $G_2$, $G_1+G_2$, is the graph whose vertex set is $V(G_1)\cup V(G_2)$ and the edge set is $E(G_1)\cup E(G_2)\cup \{uv:u\in V(G_1), v\in V(G_2)\}$. For a graph $G$, a subset of $V(G)$ is called an {\it independent set}
of $G$ if the subgraph it induces has no edges. Two edges in $G$ are said to be {\it independent edges} if they are non-adjacent.\\
A connected graph is called $c$-connected, for $c\ge 1$, if either $G$ is a complete graph $K_{c+1}$ or else it has at least $c+2$ vertices and has no $(c-1)$-vertex cut. On same lines, a graph is $c$-edge-connected if it has at least two vertices and does not contain any $(c-1)$-edge cut. The maximum value of $c$ such that a connected graph $G$ is $c$-connected is the {\it connectivity} of $G$ and denoted by $\kappa(G)$. The edge-connectivity, $\kappa'(G)$ is defined analogously. Note that for a graph $G$ of order $n$ we have $\kappa(G)\leq \kappa'(G)\leq \delta(G)\leq n-1$ and $\kappa(G)=n-1, \kappa'(G)=n-1$ and $G=K_n$ are equivalent.

Throughout this paper, $P_n, S_n, C_n$ and $K_n$ represent the path, star, cycle and complete graphs with $n$ vertices.  

For a graph $G$ without isolated vertices, Kier et. al. \cite{kh} proposed the zeroth-order Randi\'c index as 
\[^0R_{-\frac{1}{2}}(G)=\sum_{u\in V(G)}d(u)^{-1/2}\]

In 2005, Li et. al. \cite{lz1} introduced the zeroth-order general Randi\'c index by replacing the fraction $-\frac{1}{2}$ by any non-zero real number $\gamma$:
\[^0R_{\gamma}(G)=\sum_{u\in V(G)}d(u)^{\gamma}\]
In \cite{pll}, authors investigated some sharp bounds on $^0R_{\gamma}$ for unicyclic graphs with $n$ vertices and diameter $d$. Volkmann \cite{v} presented sufficient conditions for digraphs to be maximally edge connected in terms of the zeroth-order general Randi\' c index. In \cite{y} Yamaguchi obtained the trees with first three largest zeroth-order general Randi\' c indices among all the trees with given order, diameter or radius. Jamil et. al. \cite{jt} investigated the extremal graphs of $k$-generalized quasi trees for zeroth-order general Randi\'c index.
For a graph $G$ without isolated vertices, the inverse degree $ID(G)$ of $G$ is defined as 
\[ID(G)=\sum_{u\in V(G)} d(u)^{-1}\]
The inverse degree of a graph was first appeared in \cite{f}. After that a lot of work have been done on inverse degree, for details we refer \cite{hlsx,ls,m,zzl}. Xu et. al. \cite{xd} determined some upper and lower bounds on the inverse degree $ID(G)$ for a connected graph $G$ in terms of other graph parameters, such as chromatic number, clique number, connectivity or number of cut edges. They also characterized the extremal graphs. In this paper, we extended their work and investigated if the corresponding results hold for zeroth-order general Randi\'c index, their results can be viewed as corollaries of the main theorems.

\section{Preliminary Results}
 First we present some lemmas that will be useful in proving main results.
From the definition of zeroth-order general Randi\'c index for $\gamma<0$ we have the following lemma.
\begin{lemma}\label{el}
	Let G be a graph such that $vw\in E(G)$ and $y,z \in V(G)$ are nonadjacent. Then for $\gamma<0$, we have
	\begin{enumerate}
		\item $^0R_{\gamma}(G-vw)>{^0R_{\gamma}}(G)$ if $d(v),d(w)\ge2$
		\item $^0R_{\gamma}(G+xy)<{^0R_{\gamma}}(G)$ where $x$ and $y$ are non-isolated vertices in $G$.
	\end{enumerate}
\end{lemma}

\begin{lemma}\label{tl}
	For $n>2$, let G be a graph of order $n$ with $v,w\in V(G)$ such that $d(v)\ge d(w)$ and $N_G(w) \backslash N_{G}[v]=\{w_1,w_2,\ldots,w_t\}$ where $t>0$. From G we obtain a new graph $G^*=(G-\{ww_1,ww_2,\ldots,ww_t\})+\{vw_1,vw_2,\ldots,vw_t\}$. If $d(w)>t$, for $\gamma<0$ we have $^0R_{\gamma}(G^*)>^0R_{\gamma}(G)$.
\end{lemma}

\begin{proof}
	From the definition of zeroth-order general Randi\'c index, we have
	$^0R_{\gamma}(G^*)-{^0R}_{\gamma}(G)=(d(v)+t)^{\gamma}-d(v)^{\gamma}+(d(w)-t)^{\gamma}-d(w)^{\gamma}$. We deduce that $f(x)=x^{\gamma}-(x-t)^{\gamma}$ is a strictly increasing function for $x>t$ and $\gamma <0$. Since $d(v)\ge d(w)$ we have $d(v)+t> d(w)$ and this implies that $^0R_{\gamma}(G^*)-{^0R}_{\gamma}(G)=f(d(v)+t)-f(d(w))>0$.  \hspace*{4.5cm}                                                      $\square$
\end{proof}

\begin{lemma}\label{phi}
	For $n\ge x\ge 2$ and $\gamma<0$, the function
\begin{equation}
 \psi(x)=(n-x)x^{\gamma}-(n-x+1)(x-1)^{\gamma}
\end{equation}
 is a strictly increasing function.
\end{lemma}
\begin{proof}
	For given $\psi(x)$, we obtain $\psi'(x)=\gamma(n-x)x^{\gamma-1}-x^{\gamma}-\gamma(n-x+1)(x-1)^{\gamma-1}+(x-1)^{\gamma}>\gamma((n-x)x^{\gamma-1}-(n-x+1)x^{\gamma-1})-x^{\gamma}+(x-1)^{\gamma}=-\gamma x^{\gamma-1}-x^{\gamma}+(x-1)^{\gamma}>0$. Hence, for given $n\ge x\ge2$ the function $\psi(x)$ is strictly increasing.  \hspace*{7cm}                                                      $\square$
	
\end{proof}

\begin{lemma}\label{fin}
	Let $n,c$ be integer numbers, $n\geq 3$ and $1\leq c\leq n-2$. For $1\leq x\leq n-c-1$ and $-1\leq \gamma<0$, the function $f(x)= x(x+c-1)^{\gamma}+(n-c-x)(n-1-x)^{\gamma}$ is minimum for $x=1$ and $x=n-c-1$. For $c=1$ and $\gamma =-1$ we get $f(x)=2$.
\end{lemma}
\begin{proof}
	We have $f(x)=f(n-c-x)$, which implies that $x=(n-c)/2$ is a symmetry axis for the graph of this function. Its derivative equals $f'(x)=(x+c-1)^{\gamma -1}(c-1+x(\gamma +1))-(n-1-x)^{\gamma -1}(n(\gamma +1)-1-\gamma c -x(\gamma +1))$.
By the symmetry of $f$ we can only consider the case when $ x \geq (n - c)/2$. We have $f'((n-c)/2)=0$ and we shall prove that $f'(x)<0$ for $ x>(n-c)/2$. This condition is equivalent to
\begin{equation}
\left(\frac{x+c-1}{n-1-x}\right)^{\gamma -1}<\frac{(\gamma +1)(n-x)-\gamma c-1}{x(\gamma +1)+c-1}.
\end{equation}
Since $x>(n-c)/2$ and $\gamma <0$ it follows that $\left(\frac{x+c-1}{n-1-x}\right)^{\gamma -1}<\left(\frac{x+c-1}{n-1-x}\right)^{ -1}=\frac{n-1-x}{x+c-1}.$	
But $	\frac{n-1-x}{x+c-1}\leq \frac{(\gamma +1)(n-x)-\gamma c-1}{x(\gamma +1)+c-1}$ since this is equivalent to $\gamma (c-1)(2x-n+c)\leq 0$ and (2) is proved.  \hspace*{4.5cm}                                                      $\square$
\end{proof}

\begin{lemma}\label{cl}\cite{bm}
	Every c-chromatic graph has at least c vertices of degree at least $c-1$.
\end{lemma}

\section{Main Results and Discussion}
In this section, we will present our main results.
\subsection{Extremal graphs w.r.t. zeroth-order general Randi\'c index in terms of chromatic number and clique number}
Let $\complement(n,c)$ denote the set of all connected graphs having order $n$ and chromatic number $c$ and $\mho(n,c)$ the set of all connected graphs with order $n$ and clique number $c$. Hereafter, we always assume that $n_1\ge n_2\ge \ldots \ge n_c$ are positive integers with $\sum_{i=1}^c n_i=n$. A complete $c$-partite graph of order $n$ whose partite sets are of size $n_1,n_2,\ldots,n_c$, respectively, is denoted by $K_{n_1,n_2,\ldots,n_c}$. The {\it Tur\'an graph}, $T_n(c)$ is a complete $c$-partite graph of order $n$ whose partite sets differ in size by at most 1. For $c=1$, the set $\complement(n,c)$ contains a single graph $\overline{K_n}$. When $c=n$, the only graph in $\complement(n,c)$ is $K_n$. Also, we have similar results on the set $\mho(n,c)$ for $c=1$ and $c=n$.  Here we will investigate the extremal graphs in $\complement(n,c)$ and $\mho(n,c)$ w.r.t. $^0R_{\gamma}$.

\begin{lemma}\label{cb}
	Suppose that there exist two indices $i,j$ such that $i\ne j$, $1\leq i,j\leq c$ and $n_j-n_i\ge 2$. Then for $\gamma<0$ we have
	\[^0R_{\gamma}(K_{n_1,\ldots,n_i,\ldots ,n_j,\ldots,n_c})>{^0R}_{\gamma}(K_{n_1,\ldots,n_i+1,\ldots,n_j-1,\ldots,n_c})\] 
\end{lemma}
\begin{proof}
	Suppose that $i<j$. From the definition of zeroth-order general Randi\'c index we have
	\begin{align*}
	^0R_{\gamma}(K_{n_1,\ldots,n_i,\ldots ,n_j,\ldots,n_c})&-{^0R}_{\gamma}(K_{n_1,\ldots,n_i+1,\ldots,n_j-1,\ldots,n_c})=n_i(n-n_i)^{\gamma}+n_j(n-n_j)^{\gamma}\\&-(n_i+1)(n-n_i-1)^{\gamma}-(n_j-1)(n-n_j+1)^{\gamma}\\
	&=\psi(x)-\psi(y+1),
	\end{align*}
	where $x=n-n_i, y=n-n_j$ and $\psi(x)$ is given by (1). Also, we have $n_j-n_i\ge2$ which implies that $x-y\ge 2$ and $x>y+1$. By Lemma \ref{phi} $\psi(x)$ is strictly increasing, which yields
	\[^0R_{\gamma}(K_{n_1,\ldots,n_i,\ldots ,n_j,\ldots,n_c})-{^0R}_{\gamma}(K_{n_1,\ldots,n_i+1,\ldots,n_j-1,\ldots,n_c})>0,\]
	which completes the proof.  \hspace*{6cm}                                                      $\square$
\end{proof}

In this subsection we always assume that $1<c<n$ and $n=cq+r$, where $0\le r <c$, i.e., $q=\left\lfloor{\frac{n}{c}}\right\rfloor$.

\begin{thm}\label{ltheo}
	For any graph $G\in\complement(n,c)$ and $\gamma<0$, we have
	\[^0R_{\gamma}(G)\ge (c-r)\left\lfloor\frac{n}{c}\right\rfloor(n-\left\lfloor\frac{n}{c}\right\rfloor)^{\gamma}+r\left\lceil\frac{n}{c}\right\rceil(n-\left\lceil\frac{n}{c}\right\rceil)^{\gamma}\]
	and lower bound is achieved if and only if $G=T_n(c)$.
\end{thm}
\begin{proof}
	Let $G\in\complement(n,c)$ such that $G$ has the minimal zeroth-order general Randi\'c index for $\gamma<0$. From the definition of the chromatic number, $G$ has $c$ color classes and every color class is an independent set. Suppose that each  color class contains $n_i$ vertices, where $1\le i\le c$. By Lemma \ref{el} one deduces that $G$ must be a complete $c$-partite graph $K_{n_1,\ldots,n_c}$ and Lemma \ref{cb} guarantees that $G= T_n(c)$.
	 We have 
	\[^0R_{\gamma}(T_n(c))=(c-r)\left\lfloor\frac{n}{c}\right\rfloor(n-\left\lfloor\frac{n}{c}\right\rfloor)^{\gamma}+r\left\lceil\frac{n}{c}\right\rceil(n-\left\lceil\frac{n}{c}\right\rceil)^{\gamma},\]
	which completes the proof.  \hspace*{5.5cm}                                                      $\square$
\end{proof}
Further, we shall use some notation introduced in \cite{xd}.
A graph obtained by attaching $n-c$ pendant vertices to one vertex of $K_c$ is called a {\it pineapple} graph and is denoted by $PA_n(c)$. $S_n(m_1,m_2,\ldots,m_c)$ will denote a connected graph of order $n$ obtained by attaching $n-c$ pendant vertices to a complete graph $K_c$, such that $m_i$ pendant vertices are attached to the $i$th  vertex of $K_c$ for $1\le i\le c$. It follows that $\sum_{i=1}^{c}m_i=n-c$. We consider that the vertices in the clique are labeled $v_1,v_2,\ldots,v_c$. From the definition of $S_n(m_1,m_2,\ldots,m_c)$ we have $S_n(0,0,\ldots,0)= K_c$ and $S_n(n-c,0,\ldots,0)= PA_n(c)$.
In the following theorem we give an upper bound on $^0R_{\gamma}(G)$ in terms of order $n$ and chromatic number $c$ of $G$.

\begin{thm}\label{cm}
	Let $\gamma\leq -1$ then for any graph $G\in \complement(n,c)$, we have
	\[^0R_{\gamma}(G)\le n-c+(n-1)^{\gamma}+(c-1)(c-1)^{\gamma}\]
	and the equality holds if and only if $G=PA_n(c)$.
\end{thm}

\begin{proof}
	Since $G$ is connected it follows that $c\geq 2$. If $c=n$ then $G=K_n$ and the theorem is verified directly. It remains to consider the case when $2\le c\le n-1$. It follows that $n\geq 3$. Clearly, for $c=2$ $G$ is a connected bipartite graph. Moreover, if $G= S_n$ (note that the star $S_n$ coincides with $S_n(n-2,0)$ and with $PA_n(2))$, then the above equality holds and in this case $^0R_{\gamma}(G)=n-1+(n-1)^{\gamma}$. Otherwise, $G$ has at least two non-pendant  vertices. This implies $^0R_{\gamma}(G)\le 2\cdot 2^{\gamma}+n-2<(n-1)+(n-1)^{\gamma}$, because $2\cdot 2^{\gamma}-1-(n-1)^{\gamma}<2\cdot 2^{\gamma}-1\leq 0$. Hence, $G$ is not maximal if $G\neq S_n$.
\\
	We now prove the theorem for $3\le c\le n-1$. Suppose that $V(G)=\{v_1,v_2,\ldots,v_n\}$. By Lemma \ref{cl}, we can consider a set of vertices $A(G)=  \\ \{v_1,v_2,\ldots,v_c\}$, such that $d(v_i)\ge c-1$ for $1\le i\le c$. Then we have $|V(G)\backslash A(G)|>0$. If there exists $v_k\in V(G)\backslash A(G)$ such that $d(v_k)\ge 2$, then $^0R_{\gamma}(G)\leq c(c-1)^{\gamma}+2^{\gamma}+n-c-1< n-c+(n-1)^{\gamma}+(c-1)(c-1)^{\gamma}$ if and only if $(c-1)^{\gamma}+2^{\gamma}-1-(n-1)^{\gamma}<0$.
The last inequality is valid since $c\geq 3$ implies $(c-1)^{\gamma}\leq 2^{\gamma}$, which yields $(c-1)^{\gamma}+2^{\gamma}-1-(n-1)^{\gamma} \leq   2^{\gamma +1}-1-(n-1)^{\gamma}<0$, which holds because $\gamma \leq -1$.\\
This shows that $d(v_k)=1$ for any $v_k\in V(G)\backslash A(G)$, where $c+1\le k\le n$, for a graph having maximum $^0R_{\gamma}$. Since $G$ has chromatic number $c$ it follows that the subgraph of $G$ induced by $A(G)$ is $K_c$. It follows that $G$ is isomorphic to a complete graph $K_c$ with $n-c$ pendant vertices, that is, $S_n(m_1,m_2,\cdots,m_c)$ such that $\sum_{i=1}^{c}m_i=n-c$. If $S_n(m_1,m_2,\ldots,m_c)=PA_n(c)$ we are done. Otherwise, by applying Lemma \ref{tl} several times and supposing that $m_1\geq m_2\geq \ldots \geq m_c$, we get
	$$^0R_{\gamma}(G)={^0R}_{\gamma}(S_n(m_1,m_2,\ldots,m_c))<{^0R}_{\gamma}(S_n(m_1+m_c,m_2,\ldots,m_{c-1},0))$$ $$<\ldots<{^0R}_{\gamma}(S_n(n-c-m_2,m_2,0,\ldots,0))<{^0R}_{\gamma}(S_n(n-c,0,\ldots,0))$$ $$={^0R}_{\gamma}(PA_n(c)),$$
	which completes the proof.  \hspace*{6.5cm}                                                      $\square$
\end{proof}
The following result gives  upper and lower bounds on zeroth-order general Randi\'c index in terms of order $n$ and clique number $c$.

\begin{thm}
	For any graph $G\in \mho(n,c)$ and $\gamma\leq -1$ we have 
	\[(c-r)\left\lfloor{\frac{n}{c}}\right\rfloor{(n-\left\lfloor{\frac{n}{c}}\right\rfloor)}^{\gamma}+r\left\lceil{\frac{n}{c}}\right\rceil(n-\left\lceil{\frac{n}{c}}\right\rceil)^{\gamma}\le ^0R_{\gamma}(G)\le  n-c+(n-1)^{\gamma}+(c-1)(c-1)^{\gamma}     \]
The lower bound is attained if and only if $G=T_n(c)$ and the upper bound if and only if $G=PA_n(c)$.
\end{thm}

\begin{proof} The main arguments of this proof are similar to those of the proof of Theorem 3.3 from \cite{xd}.\\
	Upper bound: Let $G'\in \mho(n,c)$ having maximum zeroth-order general Randi\'c index. Since $G'$ has clique number $c$, we can assume that $G'$ contains a clique $\{v_1,v_2,\ldots,v_c\}$.
	From Lemma \ref{el} (1), we can see that $G'$ must be a graph obtained by attaching to $v_i$ some tree $T_i$ for $1\le i\le c$. Then the chromatic number of $G'$ is $c$, and the result immediately follows from the proof of Theorem \ref{cm}.
	
	Lower bound: Let $\mho'(n,c)$ be the set of graphs having order $n$ and clique number less than or equal to $c$. We shall prove the claim below first.
	
	{\bf Claim 1.} For any graph $G\in \mho'(n,c)$, we have
	\[^0R_{\gamma}(G)\ge {^0R}_{\gamma}(T_n(c))\]
	and equality holds if and only if $G= T_n(c)$.
	
	{\bf Proof of Claim 1.} If $G=K_{n_1,\ldots,n_c}$, then by Lemma \ref{cb}    we have
	\[^0R_{\gamma}(G)={^0R}_{\gamma}(K_{n_1,\ldots,n_c})\ge ^0R_{\gamma}(T_n(c))\]
	and equality holds if and only if $G= T_n(c)$.	
	
	Otherwise, $G$ is not a multipatite complete graph of the form $K_{n_1,\ldots,n_c}$. Let $u\in V(G)$ such that $u$ has maximum degree $d(u)=\triangle(G)$ in $G$. Let $A=N_G(u)$ and $B=V(G)\backslash A$. The clique number of the induced subgraph of $G$ by $A$, $G[A]$, is at most $c-1$ since $\omega(G)\le c$. Now we construct a new graph $G^*$ on the vertex set $V(G)$ as follows: $G^*$ is obtained from the subgraph $G[A]$ and the subset $B$ by joining by edges all vertices in $A$ to all vertices of $B$ and removing all possible edges which have both ends in $B$. 
	One can easily notice that $B$ is an independent set of $G^*$ and $\omega(G^*)\le c$. Let $w\in V(G^*)=V(G)$; if $w\in A$ we have $d_{G^*}(w)\ge d_G(w)$ from the construction of $G^*$ and if $w\in B$ we have $d_{G^*}(w)\ge d_G(w)$ by the choice of $u$. This implies that $^0R_{\gamma}(G^*)\le {^0R}_{\gamma}(G)$. If $G^*$ is isomorphic to a complete $t$-partite graph of order $n$, where $2\leq t \le c$, then we have $^0R_{\gamma}(G)\ge ^0R_{\gamma}(G^*)=^0R_{\gamma}(K_{n_1,\ldots,n_t})\ge ^0R_{\gamma}(T_n(t)) \ge ^0R_{\gamma}(T_n(c))$. \\
The inequality $^0R_{\gamma}(T_n(t)) \ge ^0R_{\gamma}(T_n(c))$ follows since $T_n(t)$
has vertex degrees equal to $n-\lfloor \frac{n}{t}\rfloor$ and $n-\lceil \frac{n}{t}\rceil$, which are less than or equal to the vertex degrees of $T_n(c)$.
	
	Otherwise, we repeat the above process on $G[A]$ by at most $c-2$ times (during this process if $G_i$ is isomorphic to a complete $t$-partite graph of order $n$, we stop the above process), obtaining a sequence of graphs:
	\[G=G_0,G_1,\ldots,G_r,G_{r+1},\ldots,G_{p-1},G_p=K_{n_1,\ldots,n_t}; \hspace{.3cm} t\le c\]
	such that 
	$^0R_{\gamma}(G)={^0R}_{\gamma}(G_0)\ge {^0R}_{\gamma}(G_1)\ge \ldots \ge {^0R}_{\gamma}(G_{p-1})\ge {^0R}_{\gamma}(G_p)$  $={^0R}_{\gamma}(K_{n_1,\ldots,n_t})\ge {^0R}_{\gamma}(T_n(c)).$
	
	Since  $G$ is not a multipatite complete graph of the form $K_{n_1,\ldots,n_c}$, then in the above sequence of graphs there must exist two consecutive non-isomorphic graphs $G_r$ and $G_{r+1}$ such that : $u$ being a vertex with maximum degree in $G_r$ and denoting $A=N_{G_r}(u)$ and $B=V(G_r)\backslash A$, when we transform $G_r$ to $G_{r+1}$, there must exist a vertex $w$ in $A$ or $B$ such that $d_{G_r+1}(w)> d_{G_r}(w)$. Hence,  
	\begin{align*}
	^0R_{\gamma}(G)={^0R}_{\gamma}(G_0)&\ge {^0R}_{\gamma}(G_1)\ge \cdots \ge ^0R_{\gamma}(G_r)> ^0R_{\gamma}(G_{r+1})  \ge \ldots  \\& \ge {^0R}_{\gamma}(G_{p-1})\ge  {^0R}_{\gamma}(G_p)={^0R}_{\gamma}(K_{n_1,\ldots,n_t})\ge {^0R}_{\gamma}(T_n(c))
	\end{align*}
	and the proof of the claim is complete.
	
	Consequently, we have shown that for any graph $G\in    \mho'(n,c)$, $^0R_{\gamma}(G)$ reaches its minimum in $\mho'(n,c)$, equal to $^0R_{\gamma}(T_n(c))=(c-r)\left\lfloor{\frac{n}{c}}\right\rfloor{(n-\left\lfloor{\frac{n}{c}}\right\rfloor)}^{\gamma}+r\left\lceil{\frac{n}{c}}\right\rceil(n-\left\lceil{\frac{n}{c}}\right\rceil)^{\gamma}$, only for $T_n(c)$. Note that     $\mho(n,c)   \subseteq   \mho'(n,c)$    with $T_n(c)\in  \mho(n,c)$    and our lower bound was proved.  \hspace*{4.5cm}                                                      $\square$
\end{proof}

\subsection{Extremal graphs w.r.t. zeroth-order general Randi\'c index in terms of number of cut edges}

In this subsection, we will investigate the bounds on zeroth-order general Randi\'c index in terms of number of cut edges. We shall also characterize the graphs which will provide the extremal values. Let $\Omega(n,c)$ be the set of connected graphs having order $n$ and $c>0$ cut edges. Let $C_{n-c}^c$ be a graph obtained by attaching $c$ pendant vertices to one vertex of cycle $C_{n-c}$. The {\it kite} graph, $K_n^c$, is obtained by identifying one vertex of $K_c$ with one pendant vertex of path $P_{n-c+1}$. It is easy to see that $K_n$ and $C_n$ have the minimal and maximal zeroth-order general Randi\'c index among all connected $n$-vertex graphs without any cut edge, respectively.

\begin{thm}
	Let $G\in \Omega(n,c)$ and $1\le c\le n-3$, then for $\gamma<0$ we have
	\[^0R_{\gamma}(G)\le c+(n-c-1)2^{\gamma}+(c+2)^{\gamma}\]
	and the equality holds if and only if $G= C_{n-c}^c$.
\end{thm}

\begin{proof}
	Suppose that the graph $G\in \Omega(n,c)$ has the maximum zeroth-order general Randi\'c index, for $\gamma<0$, with cut edge set $C=\{e_1,e_2,\ldots,e_c\}$. To prove the main result we first prove two claims.
	
	{\bf Claim 1.} Let $e\in C$, then $e$ must be a pendant edge.
	
	{\bf Proof of Claim 1.} On contrary suppose that $e_1=u_1v_1$ is a non-pendant edge in $G$ such that $d_G(u_1)\ge d_G(v_1)>1$. Suppose that $N_G(v_1)\backslash \{u_1\}=\{v_{11},v_{12},\ldots,v_{1t}\}$. Now we construct a new graph $$G^*=G-\{v_1v_{11},v_1v_{12},\ldots,v_1v_{1t}\}+\{u_1v_{11},u_1v_{12},\ldots,u_1v_{1t}\}.$$ Clearly, $G^*\in \Omega (n,c)$. Since $e_1$ is a cut edge in $G$, so $N_G(v_1)\backslash N_G[u_1]=\{v_{11},v_{21},\ldots,v_{t1}\}$. Then by Lemma \ref{tl} and the case when $u_1$ and $v_1$ are adjacent we have $^0R_{\gamma}(G^*)>{^0R_{\gamma}}(G)$, which contradicts the maximality of $G$.
	
	{\bf Claim 2.} The edges of $C$ have a common vertex. 
	
	{\bf Proof of Claim 2.} From above all the edges $e_i$, $1\le i\le c$ are pendant. On contrary suppose that $e_1=u_1v_1$ and $e_2=u_2v_2$ are two distinct edges in $G$ such that $d_G(u_i)=1$ for $i\in \{1,2\}$ and $v_1\ne v_2$. By applying Lemma \ref{tl} on $v_1$ and $v_2$, we obtain a new graph $G^{**}\in\Omega(n,c)$ such that $^0R_{\gamma}(G^{**})>{^0R}_{\gamma}(G)$, which is a contradiction.
	
	From Claim 2, we conclude that $G$ must be a graph obtained by attaching $c$ pendant vertices to one vertex, say $v_0$ of $G_0$ where $G_0$ is a connected graph without cut edges. Considering that in $G_0$  any vertex has degree greater or equal to 2, we have
	\[^0R_{\gamma}(G)=\sum_{w\in V(G_0)\backslash\{v_0\}}(d_{G_0}(w))^{\gamma}+(d_{G_0}(v_0)+c)^{\gamma}+c\le (n-c-1)2^{\gamma}+(c+2)^{\gamma}+c\]
	and the equality holds if and only if $d_{G_0}(v)=2$ for each $v\in G_0$, i. e., $G_0= C_{n-c}$. Equivalently, $G= C_{n-c}^c$ and we are done.  \hspace*{2.5cm}                                                      $\square$
\end{proof}

\subsection{Extremal graphs w.r.t. ${^0R}_{\gamma}$ in terms of  vertex(edge) connectivity}
Let $V_{n,c}$ and $E_{n,c}$ be the set of all graphs of order $n$ with connectivity and edge-connectivity, respectively, at most $c\le n-1$. The set of all graphs of order $n$ having connectivity and edge-connectivity equal to $c\le n-1$ is denoted by $V_n^c$ and $E_n^c$, respectively.

\begin{thm}\label{edgecon}
	For any graph $G\in V_n^c $ with $1\le c\le n-1$ and $-1\leq\gamma <0$, we have
	\[^0R_{\gamma}(G)\ge c(n-1)^{\gamma}+(n-c-1)(n-2)^{\gamma}+c^{\gamma}\]
	with equality holding if and only if $G= K_c+(K_{n_1}\cup K_{n_2})$ with $n_1,n_2\geq 1$ and $n_1+n_2=n-1$ for $c=1$ and $\gamma =-1$ and $G= K_c+(K_{1}\cup K_{n-c-1})$ for $c=1$ and $-1<\gamma<0$ or $c\ge2$.
\end{thm}

\begin{proof}
	Suppose $G\in V_n^c$ is a graph with minimal $^0R_{\gamma}(G)$ with $c$-vertex cut $S=\{v_1,v_2,\ldots,v_c\}$. By Lemma \ref{el} (2), the induced subgraph $G[S]$ is a complete graph $K_c$. 
	
	For $c=n-1$, there is a unique graph $K_n$ in the set $V_n^c$, which can be deal as a special case of $G=K_c+(K_{1}\cup K_{n-c-1})$ with $c=n-1$. So, in what follows we shall consider $1\le c\le n-2$. 
	
	{\bf Claim 1.} $G-S$ has exactly two components.
	
	{\bf Proof of claim 1.} On contrary suppose that  $G-S$ has at least three components $G_1,G_2$ and $G_3$ having $u_i\in V(G_i)$ for $i=1,2$. Then we find $G+u_1u_2\in V_n^c$, which implies $^0R_{\gamma}(G+u_1u_2)<{^0R}_{\gamma}(G) $, which contradicts the choice of $G$. Now we assume that $G-S=G_1\cup G_2$, where $G_1$ and $G_2$ are the components of $G-S$. From Lemma \ref{el} (2), we conclude that $G_1$ and $G_2$ are cliques and each vertex in $S$ is adjacent to all vertices in $G_1\cup G_2$. Consequently, we get $G=K_c+(K_{n_1}\cup K_{n_2})$ where $n_1+n_2=n-c$.
	
	Without loss of generality, assume that $n_1\le n_2$ in $G=K_c+(K_{n_1}\cup K_{n_2})$. We get
	\[^0R_{\gamma}(G)=c(n-1)^{\gamma}+n_1(n_1+c-1)^{\gamma}+n_2(n_2+c-1)^{\gamma}.\] 
	For $c=1$ and $\gamma =-1$, we have $^0R_{\gamma}(G)=\frac{1}{n-1}+2$ for any graph $G$ of the form $G= K_c+(K_{n_1}\cup K_{n_2})$ with $1\leq n_1\leq n_2\leq n-2$ and $n_1+n_2=n-1$.
	For $c\geq 2$ or $-1<\gamma <0$ we need to determine the minimal value of the following function:
	\[f(n_1,n_2)=n_1(n_1+c-1)^{\gamma}+n_2(n_2+c-1)^{\gamma},\]
where $1\leq n_1\leq n_2\leq n-c-1$ and $n_1+n_2=n-c$.
	
	From Lemma \ref{fin} we deduce that the $f(n_1,n_2)$ is minimal when $n_1=1$ and $n_2=n-c-1$ for $n_1+n_2=n-c$. Hence $^0R_{\gamma}(G)$ reaches its minimum value if and only if $G=K_c+(K_1\cup K_{n-c-1})$.  \hspace*{4.5cm}                                                      $\square$
\end{proof} 

Let $\psi(x)=x^{\gamma}+(n-x-1)(n-2)^{\gamma}+x(n-1)^{\gamma}$ for $x>0$ and $\gamma<0$, we have $\psi'(x)=\gamma x^{\gamma-1}-(n-2)^{\gamma}+(n-1)^{\gamma}<0$. This implies that $\psi(x)$ is strictly decreasing for $x>0$. Therefore, we have
\begin{equation}\label{eq}
^0R_{\gamma}(K_i+(K_1\cup K_{n-i-1}))<{^0R}_{\gamma}(K_{i-1}+(K_1\cup K_{n-i}))
\end{equation} for $2\le i\le c$.

Considering that $V_{n,c}=\cup_{i=1}^cV_{n}^i$, by Theorem \ref{edgecon} and inequality \ref{eq}, we have the following result:

\begin{thm}\label{3.7}
	For any graph $G\in V_{n,c}$ with $1\le c\le n-1$ and $-1\leq \gamma <0$, we have
	\[^0R_{\gamma}(G)\ge c^{\gamma}+(n-c-1)(n-2)^{\gamma}+c(n-1)^{\gamma}\] 
	and equality holds if and only if  $G= K_c+(K_{n_1}\cup K_{n_2})$ with $n_1,n_2\geq 1$ and $n_1+n_2=n-1$ for $c=1$ and $\gamma =-1$ and $G= K_c+(K_{1}\cup K_{n-c-1})$ for $c=1$ and $-1<\gamma<0$ or $c\ge2$.
\end{thm}

Because $\kappa(G)\le \kappa'(G)$ we have $E_n^c\subseteq V_{n,c}$. We get
$K_c + ( K_{1}\cup K_{n-c-1})\in E_{n}^{c}$ for $c\geq 1$ but $K_1+(K_{n_{1}}\cup K_{n_{2}})\notin E_{n}^{1}$ if $1<n_1\leq n_2$ with $n_1+n_2=n-1$.
From Theorem \ref{3.7} we deduce the following corollary:

\begin{cor}\label{cor3.1}
	For any graph $G\in E_n^c$ with $1\le c\le n-1$ and $-1\leq\gamma<0$, we have
	\[^0R_{\gamma}(G)\ge c^{\gamma}+(n-c-1)(n-2)^{\gamma}+c(n-1)^{\gamma}\]
	and equality holds if and only if $G= K_c+(K_1\cup K_{n-c-1})$.
\end{cor}

Again using inequality \ref{eq} and the inequalities $\kappa(G)\le \kappa'(G)\le \delta(G)$, we have the following corollary:

\begin{cor}
	Let G be any connected graph of order n and minimum degree $\delta(G)=c$, then
 for $-1\leq \gamma <0$ we have:
	\[^0R_{\gamma}(G)\ge c^{\gamma}+(n-c-1)(n-2)^{\gamma}+c(n-1)^{\gamma}\]
	with equality holding if and only if $G= K_c+(K_1\cup K_{n-c-1})$.
\end{cor}

We can see that $E_{n,c}=\cup_{i=1}^cE_{n}^{i}$. From Corollary \ref{cor3.1} and inequality (\ref{eq}), we can obtain the following result:

\begin{thm}
	For any graph $G\in E_{n,c}$ with $1\le c\le n-1$ and $-1\leq \gamma<0$, we have 
	\[^0R_{\gamma}(G)\ge c^{\gamma}+(n-c-1)(n-2)^{\gamma}+c(n-1)^{\gamma}\]
	with equality holding if and only if $G= K_c+(K_1\cup K_{n-c-1})$.	
\end{thm}

One can notice that for any edge $e\in E(G)$, where $G\in V_{n,c}$(respectively $E_{n,c}$), $G-e$ also belongs to  $V_{n,c}$(respectively $E_{n,c}$). From  \cite{lz} we know that $S_n$ has the maximal index  $^0R_{\gamma}$ among all trees of order $n$ for $\gamma<0$ since the function $\varphi(x)=(x+1)^{\gamma}-x^{\gamma}$ is strictly increasing for $x>0$ if $\gamma <0$. So from Lemma \ref{tl} (ii) we have the following consequences:
\begin{thm}
	For any graph $G\in V_{n,c}$ with $1\le c\le n-1$ and $ \gamma<0$, we have
	\[^0R_{\gamma}(G)\le (n-1)+(n-1)^{\gamma}\]
	and the equality holds if and only if $G=S_n$.
\end{thm}

\begin{thm}
	For any graph $G\in E_{n,c}$ with $1\le c\le n-1$ and $\gamma<0$, we have
	\[^0R_{\gamma}(G)\le (n-1)+(n-1)^{\gamma}\]
	and the equality holds if and only if $G=S_n$.
\end{thm}

\section{Acknowledgment}
This research is supported by the UPAR Grant of United Arab Emirates,
Al-Ain, UAE via Grant No. G00002590.

\end{document}